\documentclass{amsart} 

\usepackage[english]{babel}
\usepackage[latin1]{inputenc}

\usepackage{amsmath}
\usepackage{amsthm}
\usepackage{amssymb}
\usepackage{tikz}

\usepackage{imakeidx} 
\usepackage{appendix}
\usepackage{tikz-cd}
\usepackage{verbatim}
\usepackage{enumitem}
\usepackage{multicol}

\usepackage[backref=page]{hyperref}
\usepackage{cleveref}

\hypersetup{colorlinks=true,linkcolor=blue,anchorcolor=blue,citecolor=blue}

\usepackage{adjustbox}

\usepackage{mathtools}

\newtheorem{thm}{Theorem}[section]
\newtheorem{prop}[thm]{Proposition}
\newtheorem{lemma}[thm]{Lemma}

\newtheorem{conj}[thm]{Conjecture}

\theoremstyle{definition}

\theoremstyle{remark}
\newtheorem{rmk}[thm]{Remark}
\numberwithin{equation}{section}

\newcommand{\F}{\mathbb F}

\newcommand{\Z}{\mathbb Z}
\newcommand{\G}{\mathbb G}

\newcommand{\Spec}{\operatorname{Spec}}
\newcommand{\Br}{\operatorname{Br}}
\newcommand{\A}{\mathbb A}

\newcommand{\cl}{\overline}
\newcommand{\set}[1]{\left\{#1\right\}}
\renewcommand{\phi}{\varphi}

\newcommand{\on}[1]{\operatorname{#1}}
\newcommand{\ang}[1]{\left \langle{#1}\right \rangle}

\title[Massey Vanishing for fourfold Massey products modulo $2$]{The Massey Vanishing Conjecture for fourfold Massey products modulo $2$}

\address{Department of Mathematics\\
	University of California\\
	Los Angeles, CA 90095 \\United States of America}

\author{Alexander Merkurjev}
\email{merkurev@math.ucla.edu}

\author{Federico Scavia}
\email{scavia@math.ucla.edu}

\thanks{The first author
	was supported by the NSF grant DMS \#1801530.}

\makeatletter
\@namedef{subjclassname@2020}{\textup{2020} Mathematics Subject Classification}
\makeatother

\date{January 2023}

\subjclass[2020]{12G05; 55S30, 11E04}

\begin{document}

\begin{abstract}
	We prove the Massey Vanishing Conjecture for $n=4$ and $p=2$. That is, we show that for all fields $F$, if a fourfold Massey product modulo $2$ is defined over $F$, then it vanishes over $F$.
\end{abstract}

\maketitle

\section{Introduction}

Let $(A,\partial)$ be a differential graded ring, that is, a cochain complex equipped with a graded associative product satisfying the Leibniz rule with respect to the differential $\partial$, and let $H^*(A)$ be the cohomology ring of $A$. For all integers $n\geq 2$ and all $a_1,\dots,a_n\in H^1(A)$, one may define the $n$-fold Massey product $\ang{a_1,\dots,a_n}$: it is a certain subset of $H^2(A)$. For $n=2$, the Massey product $\ang{a_1,a_2}$ is equal to the singleton $\set{a_1a_2}$, but for $n\geq 3$ the Massey product $\ang{a_1,\dots,a_n}$ can be empty or contain more than one element. One says that $\ang{a_1,\dots,a_n}$ is defined if it is non-empty, and that it vanishes if it contains $0$. (See the introduction of \cite{harpaz2019massey} for the precise definition of Massey product, which will not be needed in this paper.) We have the following implications:
\[\text{$\ang{a_1,\dots,a_n}$ vanishes} \Rightarrow \text{$\ang{a_1,\dots,a_n}$ is defined} \Rightarrow a_ia_{i+1}=0\quad (i=1,\dots,n-1).\]
Massey \cite{massey1958higher} introduced Massey products in Algebraic Topology; in this case $A$ is the singular cochain complex of a topological space. Massey proved that the Borromean rings are not equivalent to three unlinked circles by showing that the singular cochain complex of the complement of the Borromean rings in $\mathbb{R}^3$ admits a non-trivial triple Massey product. 

In this paper, we consider Massey products in Galois cohomology. Let $p$ be a prime number, $\Gamma$ be a profinite group and $A\coloneqq C^*(\Gamma,\Z/p\Z)$ be the differential graded $\F_p$-algebra of mod $p$ continuous cochains of $\Gamma$. We write $H^*(\Gamma,\Z/p\Z)$ for the cohomology algebra $H^*(A)$. When $\Gamma$ is the absolute Galois group of a field $F$, we will write $H^*(F,\Z/p\Z)$ for $H^*(\Gamma,\Z/p\Z)$.

Let $n\geq 2$ be an integer, $U_{n+1}\subset \on{GL}_{n+1}(\F_p)$ be the subgroup of upper unitriangular matrices, and $Z_{n+1}\subset U_{n+1}$ be the subgroup generated by the matrix $Z$ having $1$ in each diagonal entry and in the entry $(1,n+1)$ and $0$ elsewhere. Then $Z_{n+1}\simeq \Z/p\Z$ is the center of $U_{n+1}$. We let $\cl{U}_{n+1}\coloneqq U_{n+1}/Z_{n+1}$; one may think of $\cl{U}_{n+1}$ as the subgroup of upper unitriangular matrices with top-right corner removed. We obtain the following diagram of groups
\[
\begin{tikzcd}
    1 \arrow[r] & \Z/p\Z \arrow[r,"\iota"] & U_{n+1} \arrow[r] \arrow[dr,"\varphi"] & \cl{U}_{n+1} \arrow[r]\arrow[d,"\cl{\varphi}"] & 1 \\
    &&& (\Z/p\Z)^n,
\end{tikzcd}
\]
where the row is a central short exact sequence, $\iota(1+p\Z)=Z$, the surjective homomorphism $\varphi$ forgets the entries of all upper diagonals of an upper unitriangular matrix except for the first one, and $\cl{\varphi}$ is induced by $\varphi$. 

Let $\chi_1,\dots,\chi_n\in H^1(\Gamma,\Z/p\Z)=\on{Hom}_{\on{cont}}(\Gamma,\Z/p\Z)$, and write $\chi$ for the group homomorphism $(\chi_1,\dots,\chi_n)\colon \Gamma\to (\Z/p\Z)^n$. Dwyer \cite{dwyer1975homology} proved that the Massey product $\ang{\chi_1,\dots,\chi_n}\subset H^2(\Gamma,\Z/p\Z)$
\begin{itemize}
    \item[--] is defined if and only if $\chi$ lifts to $\cl{U}_{n+1}$, i.e., $\chi=\cl{\varphi}\circ\chi'$ for some homomorphism $\chi'\colon \Gamma\to \cl{U}_{n+1}$, and
    \item[--] vanishes if and only if $\chi$ lifts to ${U}_{n+1}$, i.e., $\chi=\varphi\circ\chi''$ for some homomorphism $\chi''\colon \Gamma\to U_{n+1}$.
\end{itemize} 
(The reader not familiar with the general definition of Massey product may take the above as the definitions of the phrases ``the Massey product is defined'' and ``the Massey product vanishes.'') 

In contrast with the situation in Algebraic Topology, Hopkins--Wickelgren \cite{hopkins2015splitting} showed that, if $F$ is a number field, all triple Massey products in $H^*(F,\Z/2\Z)$ vanish as soon as they are defined. This result was extended to
all fields $F$ by Min\'{a}\v{c}--T\^{a}n \cite{minac2015triple, minac2016triple}. It motivated the following conjecture, known as the Massey Vanishing Conjecture, which first appeared in \cite{minac2017triple} under an assumption on roots of unity, then in general in \cite{minac2016triple}. 

\begin{conj}[Min\'{a}\v{c}--T\^{a}n]\label{massey-conj}
	For every field $F$, every prime $p$, every integer $n\geq 3$ and all $\chi_1,\dots,\chi_n\in H^1(F,\Z/p\Z )$, if the Massey product $\ang{\chi_1,\dots,\chi_n}\in H^2(F,\Z/p\Z)$ is defined, then it vanishes.
\end{conj}
When $p$ is invertible in $F$ and $F$ contains a primitive $p$-th root of unity, by Kummer Theory the characters $\chi_1,\dots,\chi_n$ correspond to scalars $a_1,\dots,a_n\in F^\times$ uniquely determined up to $p$-th powers. One says that $\ang{a_1,\dots,a_n}$ is defined (resp. vanishes) when $\ang{\chi_1,\dots,\chi_n}$ is defined (resp. vanishes). \Cref{massey-conj} then predicts that $\ang{a_1,\dots,a_n}$ vanishes as soon as it is defined.

    \Cref{massey-conj} is in the spirit of the \emph{profinite inverse Galois problem}, i.e, of the fundamental question: Which profinite groups are realizable as absolute Galois groups of fields? Indeed, a historically fruitful approach to the profinite inverse Galois problem has been to give constraints on the cohomology of absolute Galois groups. The most spectacular example of this is the Norm-Residue Theorem (the Bloch--Kato Conjecture), proved by Rost and Voevodsky.
    
    The Norm-Residue Theorem implies, in particular, that $H^*(F,\Z/p\Z)$ is a quadratic algebra: it admits a presentation with generators in degree $1$ and relations in degree $2$. This property is false in general for arbitrary profinite groups, and so gives a way to prove that a profinite group does not arise as the absolute Galois group of a field. 
    
    From this point of view,   \Cref{massey-conj} predicts a new way in which the cohomology of absolute Galois groups is simpler than that of arbitrary profinite groups. Already the $n=3$ case of \Cref{massey-conj} yields remarkable restrictions on the profinite groups which can appear as absolute Galois groups; see for example the work of Efrat \cite{efrat2014zassenhaus} and Min\'{a}\v{c}--T\^{a}n \cite{minac2017counting}.

Since its formulation, \Cref{massey-conj} has motivated a large body of work by many authors. It is known in a number of cases:
\begin{itemize}
\item[--] when $F$ is a number field, $n=3$ and $p=2$, by Hopkins--Wickelgren \cite{hopkins2015splitting};
\item[--] when $F$ is arbitrary, $n=3$ and $p=2$, by Min\'{a}\v{c}--T\^{a}n \cite{minac2015triple, minac2016triple};
\item[--] when $F$ is arbitrary, $n=3$ and $p$ is odd, by Matzri \cite{matzri2014triple}, followed by Efrat--Matzri \cite{efrat2017triple} and Min\'{a}\v{c}--T\^{a}n \cite{minac2016triple};
\item[--] when $F$ is a number field, $n=4$ and $p=2$, by Guillot--Min\'{a}\v{c}--Topaz--Wittenberg \cite{guillot2018fourfold};
\item[--] when $F$ is a number field and $n$ and $p$ are arbitrary, by Harpaz--Wittenberg \cite{harpaz2019massey}.
\end{itemize}
There are also results for specific classes of fields; for example, rigid odd fields \cite{minac2015kernel}. However, when $F$ is an arbitrary field, very little is known beyond the $n=3$ case. 

In this paper, we prove the case $n=4$ and $p=2$ of \Cref{massey-conj}, with no assumptions on $F$.

\begin{thm}\label{mainthm}
	\Cref{massey-conj} is true for $n=4$ and $p=2$. That is, for all fields $F$ and all $\chi_1,\chi_2,\chi_3,\chi_4\in H^1(F,\Z/2\Z)$, if the mod $2$ Massey product $\ang{\chi_1,\chi_2,\chi_3,\chi_4}$ is defined, then it vanishes.
\end{thm}

The proof of \Cref{mainthm} is different from those of Guillot--Min\'{a}\v{c}--Topaz--Witten\-berg and of Harpaz--Wittenberg in the number field case, as the tools used by them (local-global principles, Brauer-Manin obstruction) are not available over an arbitrary field. 

We sketch the proof of \Cref{mainthm}. If $K$ is a field (or a product of fields) of characteristic different from $2$, and $a,b\in K^\times$, we denote by $\Br(K)$ the Brauer group of $K$, and by $(a,b)\in\Br(K)$ the class of the quaternion algebra corresponding to $a$ and $b$. We also set $K_a\coloneqq K[x_a]/(x_a^2-a)$ and $K_{a,b}\coloneqq (K_a)_b$.

\begin{enumerate}
    \item  Since \Cref{massey-conj} is known when $\on{char}(F)=p$, for the proof of \Cref{mainthm} we may suppose that $\on{char}(F)\neq 2$. Let $a,b,c,d\in F^\times$ such that the mod $2$ Massey product $\ang{a,b,c,d}$ is defined: we must show that $\ang{a,b,c,d}$ vanishes. 
    \item Guillot--Min\'{a}\v{c}--Topaz--Wittenberg showed that $\ang{a,b,c,d}$ vanishes if and only if there exist $\alpha\in F_a^\times$ and $\delta\in F_d^\times$ such that $N_{F_a/F}(\alpha)=b$ in $F^\times/F^{\times 2}$, $N_{F_d/F}(\delta)=c$ in $F^\times/F^{\times 2}$ and $(\alpha, \delta)=0$ in $\Br(F_{a,d})$; see \cite[Theorem A]{guillot2018fourfold} or \Cref{u5-rephrase}(a) below. In the same spirit, we show in \Cref{u5-rephrase}(b) that $\ang{a,b,c,d}$ is defined if and only there exist $\alpha\in F_a^\times$ and $\delta\in F_d^\times$ such that $N_{F_a/F}(\alpha)=b$ in $F^\times/F^{\times 2}$, $N_{F_d/F}(\delta)=c$ in $F^\times/F^{\times 2}$ and $(\alpha, \delta)\in \Br(F_{a,d})$ comes from $\Br(F)[2]$. This reduces \Cref{mainthm} to a problem about Brauer groups.
    
    \item Fix $\alpha$ and $\delta$ as in \Cref{u5-rephrase}(b): our idea is to look for $x,y\in F^\times$ such that $(\alpha x, \delta y)=0$ in $\Br(F_{a,d})$. Note that we are allowed to replace $\alpha$ by $\alpha x$ and $\delta$ by $\delta y$: indeed, $N_{F_a/F}(\alpha x)=bx^2 =b$ in $F^\times/F^{\times 2}$ and similarly $N_{F_d/F}(\delta y)=c$ in $F^\times/F^{\times 2}$. The key new insight to find $x$ and $y$ is \Cref{x-nu}: There exist $x\in F^\times$ and $\nu \in F_a^\times$ such that $(\alpha x,\delta)=(\alpha x,\nu)$ in $\Br(F_{a,d})$ and $N_{F_a/F}(\alpha x,\nu)=0$ in $\Br(F)$. While \Cref{x-nu} is a purely algebraic statement, its proof, which takes up the whole \Cref{key-section}, is quite geometric: it consists of an intricate combination of residue computations and specialization arguments. We explain the main ideas that led us to the formulation and proof of this decisive result in \Cref{ideas}. At the end of the proof of \Cref{x-nu}, it is crucial to specialize at a particular $F$-point at which not all functions are necessarily defined; we recall how this can be done in \Cref{specialize}. 
    \item Once $x$ and $\nu$ as in (3) are found, a previous result of ours yields $y\in F^\times$ such that $(\alpha x,\nu y)=0$ in $\Br(F_a)$; see \cite[Proposition 4.4]{merkurjev2022degenerate} or \Cref{albert} below. The proof of \Cref{albert}, recalled in this paper for completeness, uses quadratic form theory, in particular, the theory of Albert forms attached to biquaternion algebras.
    \item The combination of (3) and (4) yields $x,y\in F^\times$ such that $(\alpha x,\delta y)=0$ in $\Br(F_{a,d})$. By (2), the Massey product $\ang{a,b,c,d}$ vanishes, as desired.
\end{enumerate}

We conclude this Introduction by discussing the relation between the present article and our previous work \cite{merkurjev2022degenerate}. In \cite[Theorem 1.3]{merkurjev2022degenerate}, we had proved \Cref{massey-conj} when $n=4$ and $p=2$ under the assumption that $\chi_1=\chi_4$. Therefore \Cref{mainthm} strengthens \cite[Theorem 1.3]{merkurjev2022degenerate}. However, \Cref{mainthm} does not allow us to recover \cite[Theorems 1.4 and 1.6]{merkurjev2022degenerate}: in particular, it does not allow us to give a negative answer to Positselski's question about non-formality of the continuous cochains in the presence of all roots of unity. The proofs of \Cref{mainthm} and \cite[Theorem 1.3]{merkurjev2022degenerate} are independent, except for the use of one common ingredient, namely the aforementioned \Cref{albert}.

As a corollary of \Cref{mainthm} and \cite[Corollary 3.9]{merkurjev2022degenerate}, we prove that a fourfold Massey product modulo $2$ vanishes over $F$ if and only if it is vanishes over an odd-degree field extension of $F$; see \Cref{odd-degree}. It is not at all clear how to prove this directly, and also whether a similar property should be true for $n>4$.

\subsection*{Notation}
Let $F$ be a field (more generally, a product of finitely many fields) of characteristic different from $2$. We let $F^\times$ be the group of invertible elements in $F$, $H^*(F,\Z/2\Z)$ be the Galois cohomology ring of $F$ with $\Z/2\Z$ coefficients, that is, the \'etale cohomology ring of the constant sheaf $\Z/2\Z$ on $\Spec(F)$, and $\Br(F)\coloneqq H^2(F,\mathbb{G}_{\on{m}})$ be the Brauer group of $F$. The Kummer sequence
\[1\to \mu_2\to \mathbb{G}_{\on{m}}\xrightarrow{\times 2} \mathbb{G}_{\on{m}}\to 1\]
induces isomorphisms $H^1(F,\Z/2\Z)\simeq F^{\times}/F^{\times 2}$ and $H^2(F,\Z/2\Z)\simeq \Br(F)[2]$. For all $a\in F^\times$, we let $(a)\in H^1(F,\Z/2\Z)$ be the class represented by $a$, and for all $a_1,\dots,a_n$, we let $(a_1,\dots,a_n)\coloneqq (a_1)\cup\dots\cup(a_n)\in H^n(F,\Z/2\Z)$. By \cite[Proposition 4.7.1]{gille2017central}, for all $a,b\in F^\times$ the image of $(a,b)\in H^2(F,\Z/2\Z)$ in $\Br(F)[2]$ is the Brauer class of the quaternion algebra corresponding to $a$ and $b$, and we denote it by $(a,b)\in \Br(F)$. 

If $E$ is an \'etale $F$-algebra, we write $N_{E/F}\colon E^\times \to F^\times$ for the norm homomorphism and $N_{E/F}\colon \Br(E)\to \Br(F)$ for the corestriction homomorphism. For all $A\in \Br(F)$, we write $A_{E}\in \Br(E)$ for the base change of $A$ to $E$.

For all $a_1,\dots,a_n\in F^\times$, we write $F_{a_1,\dots,a_n}$ for the \'etale $F$-algebra \[F[x_1,\dots,x_n]/(x_1^2-a_1,\dots,x_n^2-a_n),\] and we set $\sqrt{a_i}\coloneqq x_i$ for all $i=1,\dots,n$.

If $F$ is a field, an $F$-variety is a separated integral $F$-scheme of finite type. If $X$ is an $F$-variety, we denote by $F(X)$ its function field and, if $P\in X$, we denote by $O_{X,P}$ the local ring of $X$ at $P$ and by $F(P)$ the residue field of $X$ at $P$.

\section{Preliminaries}

In this section, we let $F$ be a field of characteristic different from $2$ and $\Gamma_F$ be the absolute Galois group of $F$.

\subsection{Fourfold Massey products}

\begin{prop}\label{u5-rephrase}
	Let $a,b,c,d\in F^\times$. 
	\begin{enumerate}
		\item[(a)] The Massey product $\langle a,b,c,d\rangle$ vanishes if and only if there exist $\alpha\in F_a^\times$, $\delta\in F_d^\times$ such that $N_{F_a/F}(\alpha)=b$ in $F^\times/F^{\times 2}$, $N_{F_d/F}(\delta)=c$ in $F^\times/F^{\times 2}$ and $(\alpha,\delta)=0$ in $\on{Br}(F_{a,d})$.
		\item[(b)]  The Massey product $\langle a,b,c,d\rangle$ is defined if and only if there exist $\alpha\in F_a^\times$, $\delta\in F_d^\times$ such that $N_{F_a/F}(\alpha)=b$ in $F^\times/F^{\times 2}$, $N_{F_d/F}(\delta)=c$ in $F^\times/F^{\times 2}$ and $(\alpha,\delta)\in \on{Br}(F_{a,d})$ belongs to the image of $\Br(F)[2]\to \Br(F_{a,d})[2]$.
	\end{enumerate}
\end{prop}

\begin{proof}
	(a) This is a reformulation of \cite[Theorem A]{guillot2018fourfold}. For a direct proof, see \cite[Corollary 3.12]{merkurjev2022degenerate}.
	
	(b) Let $U_5\subset \on{GL}_5(\F_2)$ be the subgroup of upper unitriangular matrices. The center of $U_5$ is isomorphic to $\Z/2\Z$ and is generated by the matrix 
	\[\begin{bmatrix}
		1 & 0 & 0 & 0 & 1 \\
		& 1 & 0 & 0 & 0 \\
		&   & 1 & 0 & 0 \\
		&   &   & 1 & 0 \\
		&   &   &   & 1
	\end{bmatrix}.\]
	Let $P$ be the normal subgroup of $U_5$ given by 
	\[P\coloneqq \begin{bmatrix}
		1 & 0 & 0 & * & * \\
		& 1 & 0 & * & * \\
		&   & 1 & 0 & 0 \\
		&   &   & 1 & 0 \\
		&   &   &   & 1
	\end{bmatrix}.\]
	Note that $P$ is the kernel of the surjective homomorphism $U_5\to U_3\times U_3$ which forgets the $2\times 2$ upper-right square of a unitriangular matrix. The center of $U_5$ is contained in $P$. If we write $\cl{U}_5$ and $\cl{P}$ for the quotient of $U_5$ and $P$ by the center of $U_5$, respectively, we obtain a commutative diagram
	\begin{equation}\label{cocycle-diag}
	\begin{tikzcd}
	& 1\arrow[d]  & 1 \arrow[d]\\
	& \Z/2\Z \arrow[r,equal] \arrow[d]  & \Z/2\Z \arrow[d] \\	
	1 \arrow[r] & P \arrow[d] \arrow[r] & U_5 \arrow[d] \arrow[r] & U_3\times U_3 \arrow[r] \arrow[d, equal] & 1 \\
	1 \arrow[r] & \cl{P}\arrow[d] \arrow[r] & \cl{U}_5\arrow[d] \arrow[r] & U_3\times U_3 \arrow[r] & 1\\
	& 1 & 1
	\end{tikzcd}
	\end{equation}
	where the rows and columns are short exact sequences. The rows endow $P$ and $\cl{P}$ with the structure of $(U_3\times U_3)$-modules, and the quotient map $P\to \cl{P}$ is $(U_3\times U_3)$-equivariant. We have a surjection $U_3\times U_3\to (\Z/2\Z)^4$, and the composition $U_5\to U_3\times U_3\to (\Z/2\Z)^4$ is the homomorphism $\varphi$ of the Introduction.
			
	Let $G$ be a finite group. We refer the reader to \cite[Definitions (18.15)]{knus1998book} for the definition of a Galois $G$-algebra. By  \cite[Example (28.15)]{knus1998book}, we have a bijection 
		\begin{equation}\label{galois-alg}
			H^1(F,G)\xrightarrow{\sim}\set{\text{Isomorphism classes of Galois $G$-algebras over $F$}}
		\end{equation}
		which is functorial in the finite group $G$ and the field $F$.
		
		Let $\alpha\in F_a^\times$, $\delta\in F_d^\times$ be such that $N_{F_a/F}(\alpha)=b$ in $F^\times/F^{\times 2}$, $N_{F_d/F}(\delta)=c$ in $F^\times/F^{\times 2}$. We may endow the \'etale algebras $(F_{a,b})_\alpha$ and $(F_{c,d})_{\delta}$ with the structures of Galois $U_3$-algebras as in \cite[\S 3.1]{merkurjev2022degenerate}. We let $h,h':\Gamma_F\to U_3$ be the group homomorphisms corresponding to the Galois $U_3$-algebras $(F_{a,b})_\alpha$ and $(F_{c,d})_{\delta}$ via (\ref{galois-alg}), respectively. Write $\chi_a,\chi_b,\chi_c,\chi_d\colon \Gamma_F\to \Z/2\Z$ for the homomorphism corresponding via  Kummer Theory to $a,b,c,d$, respectively. The composition of $(h,h')\colon \Gamma_F\to U_3\times U_3$ with the homomorphism $U_3\times U_3\to (\Z/2/\Z)^4$ is equal to $(\chi_a,\chi_b,\chi_c,\chi_d)$. Moreover, every lift of  $(\chi_a,\chi_b,\chi_c,\chi_d)$ to $U_3\times U_3$ arises in this way, for a suitable choice of $\alpha\in F_a^\times$ and $\delta\in F_d^\times$.

    \begin{lemma}\label{u3u3}
The homomorphism $(h,h')\colon\Gamma_F\to U_3\times U_3$ lifts to $\cl{U}_5$ if and only if $(\alpha,\delta)$ belongs to the image of $\Br(F)[2]\to \Br(F_{a,d})[2]$.
    \end{lemma}
  
	\begin{proof}
    We showed in \cite[Proof of Proposition 3.11]{merkurjev2022degenerate} that
		\[
		P=\on{Ind}_{F_{a,d}}^{F}(\Z/2\Z),
		\]
		where we view $P$ as a $\Gamma_F$-module via $(h,h')\colon\Gamma_F\to U_3\times U_3$, and that the class of the pullback of the middle row of (\ref{cocycle-diag}) along $(h,h')$ in $H^2(F,P)=H^2(F_{a,d},\Z/2\Z)=\on{Br}(F_{a,d})[2]$ is equal to  $(\alpha,\delta)$.
		
		The left vertical sequence in (\ref{cocycle-diag}) induces a commutative diagram
		\begin{equation}\label{brauer-exact}
		\begin{tikzcd}
		H^2(F,\Z/2\Z) \arrow[d,"\wr"]  \arrow[r] & H^2(F,P)\arrow[d,"\wr"]\arrow[r] & H^2(F,\cl{P})\arrow[d,equal]\\
		\Br(F)[2] \arrow[r] & \Br(F_{a,d})[2] \arrow[r]& H^2(F,\cl{P}),
		\end{tikzcd}
	\end{equation}
		where the rows are exact, the homomorphism $\Br(F)[2]\to \Br(F_{a,d})[2]$ is the pullback map and the homomorphism $\Br(F_{a,d})[2]\to H^2(F,\cl{P})$ is defined by the commutativity of the square on the right.

  The homomorphism $(h,h')\colon\Gamma_F\to U_3\times U_3$ lifts to a homomorphism $\Gamma_F\to \cl{U}_5$ if and only if the pullback of the middle row of (\ref{cocycle-diag}) along $(h,h')$ splits, that is, if and only if the corresponding element in $H^2(F,\cl{P})$ is trivial. By (\ref{brauer-exact}), this happens if and only if the map $\Br(F_{a,d})[2]\to H^2(F,\cl{P})$ sends $(\alpha,\delta)$ to $0$, that is, if and only if $(\alpha,\delta)$ belongs to the image of $\Br(F)[2]\to \Br(F_{a,d})[2]$.
  \end{proof}
		
		We may now complete the proof of (b). Suppose first that  the Massey product $\ang{a,b,c,d}$ is defined. By Dwyer's Theorem \cite{dwyer1975homology} (see also \cite[Theorem 2.4]{merkurjev2022degenerate}), the homomorphism $(\chi_a,\chi_b,\chi_c,\chi_d)\colon \Gamma_F\to (\Z/2\Z)^4$ lifts to a homomorphism $\Gamma_F\to\cl{U}_5$ and hence, in particular, to a homomorphism $(h,h')\colon\Gamma_F\to U_3\times U_3$. The Galois $U_3$-algebras corresponding to $h$ and $h'$ are of the form $(F_{a,b})_{\alpha}$ and $(F_{c,d})_{\delta}$ for some $\alpha\in F_a^\times$ and $\delta\in F_d^\times$ such that $N_{F_a/F}(\alpha)=b$ in $F^\times/F^{\times 2}$ and $N_{F_d/F}(\delta)=c$ in $F^\times/F^{\times 2}$, respectively. By construction $(h,h')$ lifts to $\cl{U}_5$, hence  \Cref{u3u3} implies that $(\alpha,\delta)\in \Br(F_{a,d})[2]$ comes from $\Br(F)[2]$.

  Conversely, suppose that $(\alpha,\delta)$ belongs to the image of $\Br(F)[2]\to \Br(F_{a,d})[2]$. Let $h,h'\colon\Gamma_F\to U_3$ be the homomorphisms corresponding to $(F_{a,b})_{\alpha}$ and $(F_{c,d})_{\delta}$, respectively.  By \Cref{u3u3}, the homomorphism $(h,h')$ lifts to $\cl{U}_5$. On the other hand, $(h,h')\colon\Gamma_F\to U_3\times U_3$ is a lift of $(\chi_a,\chi_b,\chi_c,\chi_d)\colon \Gamma_F\to (\Z/2\Z)^4$. This shows that $(\chi_a,\chi_b,\chi_c,\chi_d)$ lifts to $\cl{U}_5$. By Dwyer's Theorem, the Massey product $\ang{a,b,c,d}$ is defined. 
\end{proof}

\subsection{Quaternion algebras}

Recall that we suppose $\on{char}(F)\neq 2$. We begin with some standard properties of quaternion algebras.

\begin{lemma}\label{cup-norm}
	Let $a,b\in F^{\times}$. The following are equivalent:
	
	(i) $(a,b)=0$ in $\on{Br}(F)$;
	
	(ii) $b\in N_{F_a/F}(F_a^\times)$;
	
	(iii) $a\in N_{F_b/F}(F_b^\times)$.
\end{lemma}

\begin{proof}
	See \cite[Propositions 1.1.7]{gille2017central}.
\end{proof}

\begin{lemma}\label{kernel-base-change}
Let $a\in F^\times$ and $A\in \Br(F)$. Then $A_{F_a}=0$ in $\Br(F_a)$ if and only if there exists $u\in F^\times$ such that $A=(a,u)$ in $\Br(F)$.
\end{lemma}

\begin{proof}
 See \cite[Chapter XIV, Proposition 2]{serre1979local}.
\end{proof}

\begin{lemma}\label{chain-lemma}
	Let $a,b,u,v\in F^{\times}$. Then $(a,u)=(b,v)$ in $\on{Br}(F)$ if and only if there exist $n_a\in N_{F_a/F}(F_a^\times)$, $n_b\in N_{F_b/F}(F_b^\times)$ and $n_{ab}\in N_{F_{ab}/F}(F_{ab}^\times)$ such that $u=n_an_{ab}$ and $v=n_bn_{ab}$.
\end{lemma}

\begin{proof}
	See \cite[Lemma A.2]{merkurjev2022degenerate}.
\end{proof}

\begin{lemma}\label{comes-from-ac}
			Let $\rho\in F_a^\times$ and $\mu\in F^\times_b$ be such that $N_{F_a/F}(\rho)=N_{F_b/F}(\mu)$. Set $d\coloneqq \on{Tr}_{F_a/F}(\rho)+\on{Tr}_{F_b/F}(\mu)$. Suppose that $d\neq 0$. Then $(\mu,a)=(d,a)$ in $\on{Br}(F_b)$.
		\end{lemma}

\begin{proof}
See \cite[Lemma A.4(3)]{merkurjev2022degenerate}.
\end{proof}

\begin{prop}\label{theta-lemma} 
	Let $a,d\in F^\times$, $\pi,\nu\in F_a^\times$, $\rho\in F_d^\times$. 
 Suppose that $(\pi,\rho)=(\pi,\nu)$ in $\on{Br}(F_{a,d})$ and that the corestrictions of $(\pi,\rho)$ to $F_d$ and $F_{ad}$ are trivial. Then $N_{F_a/F}(\pi,\nu)=(d, n_a)$ in $\Br(F)$ for some $n_a\in N_{F_a/F}(F_a^\times)$.
\end{prop}

\begin{proof}
	We have 
	\[(N_{F_a/F}(\pi,\nu))_{F_d}=N_{F_{a,d}/F_d}(\pi,\nu)=N_{F_{a,d}/F_d}(\pi,\rho)=0 \text{ in $\Br(F_d)$,}\] and similarly $(N_{F_a/F}(\pi,\nu))_{F_{ad}}=0$ in $\Br(F_{ad})$. Thus by \Cref{kernel-base-change} there exist $f_1,f_2\in F^\times$ such that \[N_{F_a/F}(\pi,\nu)=(d,f_1)=(ad,f_2) \text{ in $\Br(F)$}.\]  
 By \Cref{chain-lemma} we have $f_1=n_an_d$, where $n_a\in N_{F_a/F}(F_a^\times)$ and $n_d\in N_{F_d/F}(F_d^\times)$.
	Since $(d,f_1)=(d,f_1n_d^{-1})$, we may replace $f_1$ by $f_1n_d^{-1}$ and hence assume that $f_1=n_a\in N_{F_a/F}(F_a^\times)$, as desired.
\end{proof}

The following proposition has already been used in the proof of \cite[Theorem 1.3]{merkurjev2022degenerate}.

\begin{prop}\label{albert}
	Let $a\in F^\times$ and $\pi,\mu\in F_a^\times$ be such that $N_{F_a/F}(\pi,\mu)=0$ in $\on{Br}(F)$. Then there exists $y\in F^\times$ such that $(\pi,\mu y)=0$ in $\on{Br}(F_a)$.
\end{prop}

\begin{proof}
	See \cite[Proposition 4.4]{merkurjev2022degenerate}. For completeness, we recall the proof. We use the theory of Albert forms attached to biquaternion algebras; see \cite[\S 16 A]{knus1998book}. 
 
	Let $s\colon F_a\to F$ be a nonzero linear map such that $s(1)=0$, let $Q$ be the quaternion algebra $(\pi,\mu)$ and let $Q^0\subset Q$ be the subspace of pure quaternions. Let $q\colon Q^0\to F_a$ be the quadratic form given by squaring: we have $q=\ang{\pi,\mu,-\pi\mu}$. Let $s_*(q)$ be the transfer of $q$; see \cite[Chapter VII, \S 1]{lam2005introduction}. Then it follows from \cite[Propositions (16.23) and (16.27)]{knus1998book} that $s_*(q)$ is similar to an Albert form over $F$ of the biquaternion $F$-algebra given by the corestriction $N_{F_a/F}(\pi,\mu)$; see the proof of \cite[Corollary (16.28)]{knus1998book}. Thus, by Albert's theorem \cite[Theorem 16.5]{knus1998book}, the fact that $N_{F_a/F}(\pi,\mu)$ is split implies that $s_*(q)$ is hyperbolic.
   
    Since $s_*(q)$ is $6$-dimensional and $4>6/2$, the $4$-dimensional subform $s_*\ang{\mu,-\pi\mu}$ of $s_*(q)$ is isotropic. We deduce that the form $\ang{\mu,-\pi\mu}$ over $F_a$ represents an element of $F$. If the form $\ang{\mu,-\pi\mu}$ is isotropic, then $\pi\in F_a^{\times 2}$, hence $(\pi,\mu)=0$ in $\on{Br}(F_a)$ and we may take $y=1$. Otherwise $\ang{\mu,-\pi\mu}$ over $F_a$ represents an element $y\in F^\times$, then $\mu y$ is represented by $\ang{1, -\pi}$. By \Cref{cup-norm}, this implies that $(\pi,\mu y)=0$ in $\on{Br}(F_a)$ and completes the proof.
\end{proof}

\subsection{Specialization}\label{specialize}
Recall from \cite[Remarks 1.11 and 2.5]{rost1996chow} that the Galois cohomology functor $H^*(-,\Z/2\Z)$ from the category of field extensions of $F$ is a cycle module, that is, it satisfies the axioms of \cite[Definitions 1.1 and 2.1]{rost1996chow}.

For all integers $n\geq 1$, all regular local $F$-algebras $R$ of dimension $n$ and all ordered systems of parameters $\pi\coloneqq (\pi_1,\dots,\pi_n)$ in $R$, letting $K$ and $K_0\coloneqq R/(\pi_1,\dots,\pi_n)$ be the fraction field and residue field of $R$, respectively, we have a specialization map
\[s_{\pi}\colon H^*(K,\Z/2\Z)\to H^*(K_0,\Z/2\Z),\]
which is a graded ring homomorphism defined as follows. 

Suppose first that $n=1$, that is, $R$ is a discrete valuation ring and $\pi=(\pi_1)$. Then we set $s_{\pi}\coloneqq \partial_{\pi_1}((-\pi_1)\cup (-))$, where $\partial_{\pi_1}\colon H^{*+1}(K,\Z/2\Z)\to H^*(K_0,\Z/2\Z)$ is the residue map at $\pi_1$; see \cite[Definition 1.1, below D4]{rost1996chow}.

Suppose now that $n\geq 2$ and that the specialization map has been defined for all regular local $F$-algebras of dimension $< n$ and all ordered systems of parameters on such algebras. For $i=2,\dots,n$ let $\cl{\pi}_i\in R/(\pi_1)$ be the reduction of $\pi_i$ modulo $\pi_1$ and set $\cl{\pi}\coloneqq (\cl{\pi}_2,\dots,\cl{\pi}_n)$: it is an ordered system of parameters in the regular local ring $R/(\pi_1)$. Then $s_{\pi}$ is defined by $s_{\pi}\coloneqq s_{\cl{\pi}}\circ  s_{(\pi_1)}$, where $\pi_1$ is viewed as an element of the localization $R_{(\pi_1)}$. 

The ring homomorphism $s_{\pi}$ depends on the choice of the ordered set $\pi$. Using the isomorphism $H^2(F,\Z/2\Z)\simeq \Br(F)[2]$ coming from Kummer Theory, we obtain a specialization map 
\[s_{\pi}\colon \Br(K)[2]\to \Br(K_0)[2].\]

Let $X$ be an $F$-variety and $P\in X$ be a regular $F$-point. For all ordered systems of parameters $\pi=(\pi_1,\dots,\pi_n)$ in the regular local ring $R=O_{X,P}$ the previous discussion yields specialization maps
\begin{equation*}
s_{P,\pi}\colon H^*(F(X),\Z/2\Z)\to H^*(F,\Z/2\Z),\quad s_{P,\pi}\colon \Br(F(X))[2]\to \Br(F)[2].
\end{equation*}
If $f\in O_{X,P}^\times$ (that is, $f$ is regular and nonzero at $P$) then it follows from the definition that $s_{P,\pi}(f)=(f(P))$. In particular, if $f\in F^\times$ is constant then $s_{P,\pi}(f)=(f)$.

\begin{lemma}\label{specialize-corestriction}
    Let $n\geq1$ be an integer, $X$ be an $n$-dimensional $F$-variety, $P\in X$ be a regular $F$-point, and $\pi\coloneqq (\pi_1,\dots,\pi_n)$ be an ordered system of parameters in $O_{X,P}$. Let $F'$ be a finite separable field extension of $F$, let $X'\coloneqq X\times_FF'$, let $P'$ be the only $F'$-point of $X'$ lying over $P$, and consider the system of parameters $\pi'\coloneqq (\pi_1\otimes 1,\dots,\pi_n\otimes 1)$ in the regular local ring $O_{X',P'}=O_{X,P}\otimes_FF'$. Then the following squares commute:
    \begin{equation*}
    \adjustbox{max width=\textwidth}{
    \begin{tikzcd}
    H^*(F(X),\Z/2\Z) \arrow[d,"(-)_{F'(X')}"] \arrow[r, "s_{P,\pi}"] & H^*(F,\Z/2\Z)\arrow[d,"(-)_{F'}"]  \\
    H^*(F'(X'),\Z/2\Z) \arrow[r,"s_{P',\pi'}"] & H^*(F',\Z/2\Z) 
    \end{tikzcd}
    \quad
    \begin{tikzcd}
    H^*(F'(X'),\Z/2\Z) \arrow[d,"N_{F'(X')/F(X)}"] \arrow[r,"s_{P',\pi'}"] & H^*(F',\Z/2\Z)\arrow[d,"N_{F'/F}"] \\
    H^*(F(X),\Z/2\Z)  \arrow[r,"s_{P,\pi}"] & H^*(F,\Z/2\Z). 
    \end{tikzcd}
    }
    \end{equation*}
\end{lemma}

\Cref{specialize-corestriction} admits an obvious generalization to the case when $F'$ is an \'etale $F$-algebra.

\begin{proof}
    We prove the result by induction on $n\geq 1$. When $n=1$, the inclusion $O_{X,P}\subset O_{X',P'}$ is an unramified extension of discrete valuation rings. Since $H^*(-,\Z/2\Z)$ is a cycle module, the commutativity of the left square then follows from  \cite[Definition 1.1, R3a]{rost1996chow}, and that of the right square from \cite[Definition 1.1, R2c and R3b]{rost1996chow}.
    
    Suppose now that $n\geq 2$. Let $Q\in X$ be the point corresponding to the prime ideal  $(\pi_1)\subset O_{X,P}$ and $Y\subset X$ be the closure of $Q$. Similarly, let $Q'\in X'$ be the point corresponding to the prime ideal  $(\pi_1\otimes 1)\subset O_{X',P'}$ and $Y'\subset X'$ be the closure of $Q'$. For $i=2,\dots,n$ let $\cl{\pi}_i\in O_{Y,P}$ be the reduction of $\pi_i$ modulo $\pi_1$, and consider the systems of parameters $\cl{\pi}\coloneqq (\cl{\pi}_2,\dots,\cl{\pi}_n)$ in $O_{Y,P}$ and $\cl{\pi}'\coloneqq (\cl{\pi}_2\otimes 1,\dots,\cl{\pi}_n\otimes 1)$ in $O_{Y',P'}=O_{Y,P}\times_FF'$. We obtain the following commutative diagram:  
    \begin{equation}\label{rectangle}
    \begin{tikzcd}
    H^*(F'(X'),\Z/2\Z) \arrow[r,"s_{(\pi_1')}"]\arrow[d,"N_{F'(X')/F(X)}"] & H^*(F'(Y'),\Z/2\Z) \arrow[d,"N_{F'(Y')/F(Y)}"] \arrow[r,"s_{P',\cl{\pi}'}"] & H^*(F',\Z/2\Z) \arrow[d,"N_{F'/F}"] \\
    H^*(F(X),\Z/2\Z)  \arrow[r, "s_{(\pi_1)}"] & H^*(F(Y),\Z/2\Z)\arrow[r,"s_{P,\cl{\pi}}"] & H^*(F,\Z/2\Z). 
    \end{tikzcd}
    \end{equation}
    Indeed, since $\dim(Y)=n-1$, the right square in (\ref{rectangle}) commutes by the inductive assumption. Moreover, $O_{X,Q}$ is a discrete valuation ring and $O_{X',Q'}=O_{X,Q}\otimes_FF'$, hence the extension $O_{X,Q}\subset O_{X',Q'}$ is an unramified extension of discrete valuation rings. The commutativity of the left square in (\ref{rectangle}) thus follows from \cite[Definition 1.1, R2c and R3b]{rost1996chow}. By definition, the composition of the top (resp. bottom) horizontal arrows in (\ref{rectangle}) is equal to $s_{(P',\pi')}$ (resp. $s_{(P,\pi)}$), therefore the second square in the statement of \Cref{specialize-corestriction} commutes. The proof of the commutativity of the first square in the statement of \Cref{specialize-corestriction} is entirely analogous, using  \cite[Definition 1.1, R3a]{rost1996chow} instead of \cite[Definition 1.1, R2c and R3b]{rost1996chow}.
\end{proof}

\section{The key proposition}\label{key-section}
The purpose of this section is the proof of the following proposition.

\begin{prop}\label{x-nu}
	Let $F$ be a field of characteristic different from $2$. Let $a,c,d\in F^\times$, let $\alpha\in F_a^\times$, and let $\delta\in F_d^\times$ be such that and $N_{F_d/F}(\delta)=c$. Suppose that $c$ is not a square in $F$ and that $(\alpha,\delta)\in \on{Br}(F_{a,d})$ is in the image of the pullback map $\on{Br}(F)[2]\to \on{Br}(F_{a,d})[2]$. Then there exist $x\in F^\times$ and $\nu \in F_a^\times$ such that:
	\begin{enumerate}
		\item $(\alpha x,\delta)=(\alpha x,\nu)$ in $\Br(F_{a,d})$, and
		\item $N_{F_a/F}(\alpha x,\nu)=0$ in $\Br(F)$.
	\end{enumerate}
\end{prop}

\begin{proof}
The proof will require several intermediate steps. To begin with, let $u_1,u_2\in F$ be such that $\delta=u_1+u_2\sqrt{d}$. We have
\begin{equation}\label{ff0}
	u_1^2-du_2^2=N_{F_d/F}(\delta)=c.
\end{equation}
Let $A\in \on{Br}(F)[2]$ be such that $A_{F_{a,d}}=(\alpha,\delta)$ in $\on{Br}(F_{a,d})$. Then
\[(\alpha,c)=N_{F_{a,d}/F_a}(\alpha,\delta)=2A_{F_a}=0\text{ in $\Br(F_{a})$}.\] It follows from \Cref{cup-norm} that there exist $\alpha_1,\alpha_2\in F_a$ such that
\begin{equation}\label{alpha-alpha1-alpha2}
\alpha=\alpha_1^2-c \alpha_2^2.
\end{equation}

\begin{lemma}\label{indep}
In order to prove \Cref{x-nu}, we may assume that $\alpha_1$ and $\alpha_2$ are linearly independent over $F$.
\end{lemma}

\begin{proof}
    Suppose that $\alpha_1$ and $\alpha_2$ are linearly dependent over $F$, so that there exists $t\in F$ such that either $\alpha_1=t\alpha_2$ or $\alpha_2=t\alpha_1$. By (\ref{alpha-alpha1-alpha2}), in the first case $\alpha=(t^2-c)\alpha_2^2$, and in the second case $\alpha=(1-ct^2)\alpha_1^2$. Thus, there exist $i\in\set{1,2}$ and $u\in F^\times$ such that $\alpha= u\alpha_i^2$. Note that $u\in F^\times$ and $\alpha_i\in F_a^\times$ because $\alpha\in F^\times_a$. Letting $x=u$ and $\nu=1$, we have $(\alpha x,\delta)=(u^2,\delta)=0$ in $\Br(F_{a,d})$ and $(\alpha x, \nu)=(ux, \nu)=0$ in $\Br(F_a)$, hence (1) and (2) of \Cref{x-nu} are satisfied.
    \end{proof}

In view of \Cref{indep}, from now on we assume that $\alpha_1$ and $\alpha_2$ are linearly independent over $F$. 

\begin{lemma}\label{2by2-matrix}
    The elements $\alpha_1+u_1\alpha_2$ and $u_1\alpha_1 + c \alpha_2$ in $F_a$ are linearly independent over $F$.
\end{lemma}

\begin{proof}
Since $\alpha_1$ and $\alpha_2$ are linearly independent over $F$, it is enough to show that the matrix
\[
\begin{bmatrix}
	1 & u_1  \\
	u_1 & c
\end{bmatrix}
\]
is invertible, i.e., that $c\neq u_1^2$. This is true because $c$ is not a square in $F$.
\end{proof}

Consider the field $K\coloneqq F(x_1,x_2)$, where $x_1$ and $x_2$ are algebraically independent over $F$. Set 
\[
f\coloneqq x_1^2-c x_2^2\in K^\times
\]
and
\[h_1\coloneqq \alpha_1x_1+c\alpha_2x_2\in K_a^\times,\qquad h_2=\alpha_1x_2+\alpha_2x_1\in K_a^\times.\]
The fact that $h_1,h_2\in K_a^\times$ follows from the linear independence of $\alpha_1$ and $\alpha_2$ over $F$.
We have
\begin{equation}\label{ff1}
	\alpha f=(\alpha_1^2-c \alpha_2^2)(x_1^2-c x_2^2)=(\alpha_1 x_1+c \alpha_2 x_2)^2 -c(\alpha_1 x_2+ \alpha_2 x_1)^2=h_1^2-ch_2^2,
\end{equation}
hence
\begin{equation}\label{ff2}
	c=\left(\frac{h_1}{h_2}\right)^2-\alpha f\left(\frac{1}{h_2}\right)^2=N_{(K_a)_{\alpha f}/K_a}(\rho),
\end{equation}
where 
\[\rho\coloneqq \frac{h_1}{h_2}+\frac{\sqrt{\alpha f}}{h_2}\in (K_a)_{\alpha f}^\times.\]
We have $\rho\in (K_a)_{\alpha f}^\times$ because $c\in F^\times$.  Define
\[h\coloneqq h_1+u_1h_2\in K_a^\times.\]
The fact that $h\in K_a^\times$ follows from \Cref{2by2-matrix}. Finally, set
\[g\coloneqq \on{Tr}_{K_{a,d}/K_a}(\delta)+\on{Tr}_{(K_a)_{\alpha f}/K_a}(\rho)
	= 2u_1+2\frac{h_1}{h_2}
	= \frac{2h}{h_2}\in K_a^\times.\]

\begin{lemma}\label{cores-trivial} 
    (a) We have $(\alpha f,\delta)=(\alpha f,g)$ in $\Br(K_{a,d})$.

    (b) There exists $h'\in K_a^\times$ such that 
\[N_{K_a/K}(\alpha f,g)=(d,N_{K_a/K}(h')) \text{ in $\Br(K)$.}\]
\end{lemma}
\begin{proof}
    (a) Since $N_{F_d/F}(\delta)=c$, we have $N_{K_{a,d}/K_a}(\delta)=c$. On the other hand, by (\ref{ff2}) we know that $N_{(K_a)_{\alpha f}/K_a}(\rho)=c$. The conclusion now follows from \Cref{comes-from-ac}, applied to the base field $K_a$ and the elements $\rho\in (K_a)_{\alpha f}^\times$ and $\delta \in K_{a,d}^\times$.

    (b) By (a) and \Cref{theta-lemma}, it is enough to show that the corestrictions of $(\alpha f,\delta)\in \Br(K_{a,d})$ to $K_d$ and $K_{ad}$ are zero. We have $(\alpha f,\delta)=(\alpha,\delta)+(f,\delta)$ in $\Br(K_{a,d})$, and so it suffices to show that the corestrictions of $(\alpha,\delta)$ and $(f,\delta)$ from $K_{a,d}$ to $K_d$ and $K_{ad}$ are trivial. 

Recall that $(\alpha,\delta)=A_{F_{a,d}}$ in $\Br(F_{a,d})$ for some $A\in \Br(F)[2]$. Thus 
\[N_{F_{a,d}/F_d}(\alpha,\delta)=N_{F_{a,d}/F_d}(A_{F_{a,d}})=2A_{F_d}=0 \text{ in $\Br(F_d)$.}\]
Similarly, $N_{F_{a,d}/F_d}(\alpha,\delta)=0$ in $\Br(F_{ad})$. Since corestrictions commute with base change, it follows that the corestrictions of $(\alpha,\delta)\in \Br(K_{a,d})$ to $K_d$ and $K_{ad}$ are zero. 

Since $\delta \in F_d^\times$, we have $N_{K_{a,d}/K_d}(f,\delta)=(f,\delta^2)=0$ in $\Br(K_d)$. Moreover, by definition $f\in N_{K_c/K}(K_c^\times)$, hence by \Cref{cup-norm} we have $N_{K_{a,d}/K_{ad}}(f,\delta)=(f,c)=0$ in $\Br(K_{ad})$. This shows that the corestrictions of $(f,\delta)\in \Br(K_{a,d})$ to $K_d$ and $K_{ad}$ are zero and completes the proof.
\end{proof}

We view $K$ as the function field of the affine plane $\A^2_F$, with coordinates $x_1$ and $x_2$. We consider the following divisors of $\A^2_{F_a}$:

\begin{enumerate}
	\item[(i)] the divisor $D_1\subset \A^2_{F_a}$ given by $f=0$. We have $D_1=D'_1\times_FF_a$, where $D'_1\subset \A^2_F$ is given by $f=0$; 
	\item[(ii)] the divisor $D_2\subset \A^2_{F_a}$ given by $h_2=0$; 
	\item[(iii)] the divisor $D_3\subset \A^2_{F_a}$ given by $h=0$. 
\end{enumerate}
The divisors $D_2$ and $D_3$ are irreducible. Since $c$ is not a square in $F$, the divisor $D_1'$ is irreducible. If $c$ is not a square in $F_a$, the divisor $D_1$ is also irreducible. If $c=c_1^2$ for some $c_1\in F_a^\times$, then $D_1=D_1^+\cup D_1^-$, where $D_1^+$ is given by the equation $x_1=c_1x_2$ and $D_1^-$ is given by the equation $x_1=-c_1x_2$.

For $i=1,2,3$, we denote by $F_a(D_i)$ the product of the residue fields of the generic points of $D_i$. We also write $F(D_1')$ for the residue field of $D_1'$. The \'etale $F(D_1')$-algebra $F_a(D_1)$ has degree $2$, and is split if and only if $c$ is a square in $F_a$. 

\begin{lemma}\label{unramified}
    (a) The Brauer class \[B\coloneqq (\alpha f, g)+(d,h)\in \Br(K_a)\] is unramified away from $D_1$, and the residue of $B$ at $D_1$ is equal to $2u_1+ 2\frac{x_1}{x_2}\in F_a(D_1)^\times /F_a(D_1)^{\times 2}$.

    (b) There exists $C$ in the image of $\Br(F)\to \Br (K)$ such that    
    \begin{equation*}
	N_{K_a/K}(\alpha f,g)=(d,N_{K_a/K}(h))+C \text{ in $\Br(K)$.}
\end{equation*}
\end{lemma}

\begin{proof}
(a) We will prove (a) by computing the residues of $(\alpha f, g)$ and $(d,h)$ at all divisors of $\A^2_{F_a}$. It is clear that $(d,h)$ is unramified away from $D_3$, and that its residue at $D_3$ is equal to $d$.

\begin{enumerate}
	\item[(i)] The rational function $\frac{x_1}{x_2}$ on $\A^2_F$ is regular and non-vanishing at the generic point of $D'_1$. We also denote by $\frac{x_1}{x_2}\in F(D_1')^\times$ the reduction of $\frac{x_1}{x_2}$. Since $f=0$ in $F(D_1')$, we have  $(\frac{x_1}{x_2})^2=c$ in $F(D_1')$. Moreover, the residue of $(\alpha f,g)$ at $D_1$ is equal to
	\[
	2u_1+2\frac{\alpha_1 \frac{x_1}{x_2}+c \alpha_2}{\alpha_1+ \alpha_2\frac{x_1}{x_2}}=2u_1+2\frac{x_1}{x_2}.\]
	\item[(ii)] It follows from (\ref{ff1}) that $\alpha f=h_1^2$ is a square
	in $F_a(D_2)$, and hence $(\alpha f,g)$ is unramified at $D_2$. Since $(d,h)$ is also unramified at $D_1$, we conclude that $B$ is unramified at $D_2$.
	\item[(iii)] In $F_a(D_3)$ we have $h=0$, hence $u_1h_2=-h_1$. It now follows from (\ref{ff0}) and (\ref{ff1}) that
	\[
	\alpha f = h_1^2-ch_2^2 = u_1^2h_2^2 -ch_2^2=(u_1^2-c)h_2^2=du_2^2h_2^2,
    \]
	that is, $\alpha f=d$ in $F_a(D_3)^\times / F_a(D_3)^{\times 2}$. Thus, the residue of $(\alpha f,g)$ at $D_3$ is equal to $d$. The residue of $(d,h)$ at $D_3$ is also equal to $d$, therefore $B$ is unramified at $D_3$.
	\item[(iv)] If $D\subset \A^2_{F_a}$ is a prime divisor different from the $D_i$ (and also $D^+$ and $D^-$ if $c$ is a square in $F_a$) and $D'\subset \A^2_F$ is the image of $D$, then $(\alpha f,g)$ is unramified at $D$, hence $N_{K_a/K}(\alpha f, g)$ is unramified at $D'$.
\end{enumerate}
This proves (a). 

(b) By the compatibility of the residue maps with corestrictions, $N_{K_a/K}(B)$ is unramified away from $D'_1$. Moreover, since $D_1=D'_1\times_FF_a$, the residue of $N_{K_a/K}(\alpha f,g)$ at $D'_1$ is equal to $N_{F_a(D_1)/F(D_1')}(2u_1+2\frac{x_1}{x_2})=(2u_1+2\frac{x_1}{x_2})^2$, where we have used the fact that $\frac{x_1}{x_2}\in F(D_1')$. Therefore the residue of $N_{K_a/K}(B)$ at $D'_1$ is trivial, that is, $N_{K_a/K}(\alpha f,g)$ is unramified at $D'_1$. The class $N_{K_a/K}(d,h)=(d,N_{K_a/K}(h))$ is also unramified at $D'_1$, hence $N_{K_a/K}(B)$ is also unramified at $D'_1$. This shows that the Brauer class $N_{K_a/K}(B)\in \Br(K)$ is unramified on $\A^2_F$. Therefore, by homotopy invariance of $\Br(-)[2]$ \cite[Proposition 2.2]{rost1996chow}, $N_{K_a/K}(B)$ comes from $\Br(F)$, as desired.
\end{proof}

The combination of \Cref{cores-trivial}(b) and \Cref{unramified}(b) yields the following result.

\begin{lemma}\label{unramified-combine}
    There exists $\eta\in F_a^\times$ such that 
    \[
	N_{K_a/K}(\alpha f,g)=(d,N_{K_a/K}(h\eta)) \text{ in $\Br(K)$.}
\]
\end{lemma}

\begin{proof}
    Let $P$ be an $F$-point of $\A^2_F$ and $\pi=(\pi_1,\pi_2)$ be a system of parameters. (If $P=(P_1,P_2)$, we may take $\pi_1=x_1-P_1$ and $\pi_2=x_2-P_2$.)  We specialize the identities of \Cref{unramified}(b) and \Cref{cores-trivial}(b) at $P$. From \Cref{specialize-corestriction} and the fact that the specialization map is a ring homomorphism we deduce that 
$C=(d,N_{K_a/K}(\eta))$ for some $\eta\in F_{a}^\times$, and hence $N_{K_a/K}(\alpha f,g)=(d,N_{K_a/K}(h\eta))$ in $\Br(K)$, as desired.
\end{proof}

The proof of \Cref{x-nu} will follow from \Cref{cores-trivial}(a) and \Cref{unramified-combine} by a specialization argument. The point at which specialization will occur is determined by the following lemma (with $\xi=\eta$).

\begin{lemma}\label{exists-f-point}
    For every $\xi \in F_a$, there exists a unique $F$-point $P_{\xi}\in \A^2_F$ such that, letting $P'_\xi$ be the base change of $P_{\xi}$ to $F_a$, we have $h(P'_{\xi})=\xi$.
\end{lemma}

\begin{proof}
Let $P\coloneqq (P_1,P_2)\in \A^2_F$ be an $F$-point, and let $P'$ be the base change of $P$ to $F_a$. The equation $h(P')=\xi$ reads
\[u_1(\alpha_1 P_2+ \alpha_2 P_1)+(\alpha_1 P_1+c \alpha_2 P_2)=\xi \text{ in $F_a$,}\]
or equivalently
\begin{equation}\label{solvable-if-indip}(\alpha_1+u_1\alpha_2 )P_1 + (u_1\alpha_1 + c \alpha_2)P_2=\xi \text{ in $F_a$.}\end{equation}
We know from \Cref{2by2-matrix} that $\alpha_1 + u_1\alpha_2$ and $u_1\alpha_1 + c \alpha_2$ are linearly independent over $F$, hence (\ref{solvable-if-indip}) has a unique solution $(P_1,P_2)\in F\times F$. 
\end{proof}

We are now in a position to prove \Cref{x-nu}. By \Cref{exists-f-point}, there exists an $F$-point $P\in \A^2_F$ such that $h(P')=\eta$, where $P'$ is the base change of $P$ to $F_a$. Let $\pi=(\pi_1,\pi_2)$ be an ordered system of parameters of the local ring at $P$. (As in the proof of \Cref{unramified-combine}, if $P=(P_1,P_2)$, then we may take $\pi_1=x_1-P_1$ and $\pi_2=x_2-P_2$.) Let $\pi'=(\pi'_1,\pi'_2)$ be the system of parameters obtained by base change of $\pi$ to $F_a$. We also let $x\in F^\times$ and $\nu\in F_a^\times$ be such that $(x)= s_{P,\pi}(f)$ and $(\nu)=s_{P',\pi'}(g)$. We obtain the following string of equalities in $\Br(F)$:
\begin{align*}
N_{F_a/F}(\alpha x,\nu)&=s_{P',\pi'}(N_{K_a/K}(\alpha f,g))\\
&=s_{P',\pi'}(d,N_{K_a/K}(h\eta))\\
&=(d,N_{F_a/F}(h(P')\eta))\\
&=(d,N_{F_a/F}(\eta)^2)\\
&=0.
\end{align*}
Here the first and third equalities follow from \Cref{specialize-corestriction}, the second equality follows from \Cref{unramified-combine} and the fourth equality from the fact that $h(P')=\eta$. On the other hand, \Cref{specialize-corestriction} and \Cref{cores-trivial}(a) yield 
\[
(\alpha x,\delta)=(\alpha x,\nu)\text{ in $\Br(F_{a,d})$.}
\]
This completes the proof.   
\end{proof}

\begin{rmk}\label{ideas}
    In the proof of \Cref{x-nu}, we chose $f$ to be a norm from $K_c^\times$, hence $x$, being obtained by specialization from $f$, belongs to $N_{F_c/F}(F_c^\times)$. Since this choice is quite surprising and at the same time crucial for our argument, we explain the reasoning that led us to it.

    We first recall our general strategy, based on \Cref{u5-rephrase}: we start with $(\alpha, \delta)\in\Br(F_{a,d})$ which belongs to the image of $\Br(F)[2]\to \Br(F_{a,d})[2]$, and we want to find $x,y\in F^\times$ so that $(\alpha x, \delta y)=0$ in $\Br(F_{a,d})$.

    If no restrictions on $x,y$ are made, it will often happen that $(\alpha x, \delta y)\in \Br(F_{a,d})$ no longer comes from $\Br(F)[2]$. Indeed, if $(\alpha x, \delta y)\in \Br(F_{a,d})$ comes from $\Br(F)[2]$, then its corestrictions to $F_a$, $F_{d}$ and $F_{ad}$ are trivial (the converse is also true). It turns out that the latter condition is equivalent to the existence of $n_a\in N_{F_a/F}(F_a^\times)$, $n_b\in N_{F_b/F}(F_b^\times)$, $n_c\in N_{F_c/F}(F_c^\times)$, $n_d\in N_{F_d/F}(F_d^\times)$, $n_{ac}\in N_{F_{ac}/F}(F_{ac}^\times)$ and $n_{bd}\in N_{F_{bd}/F}(F_{bd}^\times)$ such that
    \begin{equation}\label{system-norms}x= n_c n_{ac},\qquad y=n_b n_{bd},\qquad n_{ac}n_{bd}=n_an_d.\end{equation}
If the solutions to (\ref{system-norms}) could be rationally parametrized, one could then plug the parametrization into $(\alpha x,\delta y)$ and hope to find $x$ and $y$ this way. The condition $n_{ac}n_{bd}=n_an_d$, however, determines a non-rational variety. Our hope then becomes to at least parametrize a suitable subset of the solutions of (\ref{system-norms}): this subset should be large enough as to contain our objective $x$ and $y$, but also small enough as to allow a simple parametrization.

In \cite[Theorem 1.3]{merkurjev2022degenerate}, where we solved the case when $ad$ is a square, we avoided this problem by choosing $x=1$, $y=n_b$ (so that $N_{F_a/F}(\alpha,\delta y)=N_{F_a/F}(\alpha,y)=(b,y)=0$) and $n_a=n_c=n_d=n_{ac}=n_{bd}=1$.

In the general case, after several tries, we chose $x=n_c$, $y=n_b$ and $n_a=n_d=n_{ac}=n_{bd}=1$. Therefore $x$ is of the form $x_1^2-cx_2^2$, and we need to find a condition for the triviality of $(\alpha x, \delta y)$ in $\Br(F_{a,d})$. It is pure luck that this condition is a system of linear equations and that this system can be solved.
\end{rmk}

\section{Proof of Theorem \ref{mainthm}}

\begin{proof}[Proof of \Cref{mainthm}]
	If $\on{char}(F)=2$ then by \cite[Theorem 9.1]{galois2002koch} the maximal pro-$2$-quotient of the absolute Galois group of $F$ is free, hence by \cite[Remark 4.1]{minac2017triple} all mod $2$ Massey products over $F$ are defined and vanish. We may thus assume that $\on{char}(F)\neq 2$, so that by Kummer Theory the characters $\chi_1,\chi_2,\chi_3,\chi_4$ correspond to scalars $a,b,c,d\in F^\times$, uniquely determined up to nonzero squares.

    Let $a,b,c,d\in F^\times$, and suppose that the Massey product $\ang{a,b,c,d}$ is defined. Suppose first that $c$ is a square in $F$, and let $\delta=1$. Then $N_{F_d/F}(\delta)=1=c$ in $F^\times/F^{\times 2}$. Moreover, since $\ang{a,b,c,d}$ is defined we know that $(a,b)=0$ in $\Br(F)$, hence by \Cref{cup-norm} there exists $\alpha\in F_a^\times$ such that $N_{F_a/F}(\alpha)=b$. We have $(\alpha,\delta)=0$ in $\Br(F_{a,d})$, and so by \Cref{u5-rephrase}(a) the Massey product $\ang{a,b,c,d}$ vanishes.

    Suppose now that $c$ is not a square in $F$. By \Cref{u5-rephrase}(b) we know that there exist $\alpha\in F_a^\times$ and $\delta\in F_d^\times$ such that $N_{F_a/F}(\alpha)=b$ in $F^\times/F^{\times 2}$ and $N_{F_d/F}(\delta)=c$ in $F^\times/F^{\times 2}$ and such that $(\alpha,\delta)$ belongs to the image of $\Br(F)[2]\to\Br(F_{a,d})[2]$. Let $v\in F^\times$ be such that $N_{F_d/F}(\delta)=cv^2$. The Massey product $\ang{a,b,c,d}$ is defined (resp. vanishes) if and only if $\ang{a,b,cv^2,d}$ is defined (resp. vanishes), hence we may replace $c$ by $cv^2$ and thus assume that $N_{F_d/F}(\delta)=c$. Since $c$ is not a square in $F$, the assumptions of \Cref{x-nu} are now satisfied, hence there exist $x\in F^\times$ and $\nu\in F_a^\times$ such that 	
	\begin{enumerate}
		\item $(\alpha x,\delta)=(\alpha x,\nu)$ in $\Br(F_{a,d})$ and
		\item $N_{F_a/F}(\alpha x,\nu)=0$ in $\Br(F)$.
	\end{enumerate} 
By \Cref{albert}, (2) implies the existence of $y\in F^\times$ such that
	$(\alpha x,\nu y)=0$ in $\Br(F_a)$. Therefore by (1) we have $(\alpha x,\delta)=(\alpha x,\nu)=(\alpha x,y)$ in $\Br(F_{a,d})$,	and hence $(\alpha x,\delta y)=0$ in $\Br(F_{a,d})$. By \Cref{u5-rephrase}, this means that the Massey product $\ang{a,b,c,d}$ vanishes, as desired.
\end{proof}

\begin{rmk}\label{odd-degree}
    (1) Let $F$ be a field of characteristic different from $2$ and $a,b,c,d\in F^\times$. In \cite[Corollary 3.9]{merkurjev2022degenerate}, we proved that the mod $2$ Massey product $\ang{a,b,c,d}$ is defined over $F$ if and only if it is defined over some odd-degree field extension of $F$. Combining this with \Cref{mainthm}, we deduce that $\ang{a,b,c,d}$ vanishes over $F$ if and only if it vanishes over some odd-degree field extension of $F$. 

    (2) Recall that a splitting variety for $\ang{a,b,c,d}$ is an $F$-variety $X$ such that for all field extensions $K/F$, $\ang{a,b,c,d}$ vanishes over $K$ if and only if $X(K)$ is not empty. Let $\alpha\in F_a^\times$ and $\delta\in F_d^\times$ be such that $N_{F_a/F}(\alpha)=b$ and $N_{F_d/F}(\delta)=c$. Consider the $F$-variety $\G_{\on{m}}^2\times R_{F_{a,d}/F}(\A^2_{F_{a,d}})$, where $R_{F_{a,d}/F}(-)$ denotes the Weil restriction functor, $\G_{\on{m}}^2$ has coordinates $(s,t)$ and $R_{F_{a,d}/F}(\A^2_{F_{a,d}})$ has coordinates $(u,v)$. Let $X\subset \G_{\on{m}}^2\times R_{F_{a,d}/F}(\A^2_{F_{a,d}})$ be the subvariety defined by the equation $s\alpha u^2+t\delta v^2=1$: according to  \cite[Below Theorem A.1]{guillot2018fourfold}, it is $6$-dimensional, smooth and geometrically rational. It is known that $X$ is a splitting variety for $\ang{a,b,c,d}$; see \cite[Theorem 5.6 and A.1]{guillot2018fourfold}. It follows from (1) that if $X$ has a zero-cycle of odd degree, then $X$ has an $F$-point. We do not know how to prove this property directly, without recourse to \Cref{mainthm}. (One should note that the equation $s\alpha u^2+t\delta v^2=1$ is really a system of four equations over $F$.)
\end{rmk}

\end{document}